\documentclass[reqno,12pt]{amsart}

\usepackage{amsfonts, amssymb}

\usepackage[latin1]{inputenc} 
\usepackage{todonotes}
\usepackage{tikz}
\usepackage{pgfplots}
\usepackage{colortbl}
\usepgfplotslibrary{fillbetween}
\usetikzlibrary{shapes,arrows}
\usepackage{graphicx}
\usepackage{a4wide}
\usepackage{dsfont}
\usepackage{xcolor}
\usepackage{cite}
\usepackage{subfigure}
\usepackage{booktabs}

\newtheorem{theorem}{Theorem}[section]                                          
\newtheorem{proposition}[theorem]{Proposition}                          
\newtheorem{lemma}[theorem]{Lemma}
\newtheorem{corollary}[theorem]{Corollary}

\newtheorem{remark}{Remark}[section]

\newcommand{\Z}{\mathbb{Z}}
\newcommand{\lc}{\lambda_c}
\newcommand{\cY}{\mathcal{Y}}
\newcommand{\cZ}{\mathcal{Z}}
\newcommand{\cX}{\mathcal{X}}
\newcommand{\ind}{\mathds{1}}

\makeatletter
\@addtoreset{equation}{section}
\makeatother

\title[A stochastic model for seed dispersal]{
On the role of reduced habitat in the phase transition of a stochastic model for seed dispersal}

\author{Cristian F. Coletti, Nevena Mari\'c and Pablo M. Rodriguez}

\date{}

\address{
\newline
Cristian F. Coletti
\newline
Centro de Matem\'atica, Computa\c{c}\~ao e Cogni\c{c}\~ao, Universidade Federal do ABC
\newline
Avenida dos Estados 5001, Bangu, Santo Andr\'e, S\~ao Paulo, Brazil
\newline
e-mail:  cristian.coletti@ufabc.edu.br
\newline
\newline
Nevena Mari\'c
\newline
School of Computing, Union University
\newline
Kneza Mihaila 6, Belgrade, Serbia
\newline
email: nmaric@raf.rs
\newline
\newline
Pablo M. Rodriguez
\newline
 Centro de Ci\^encias Exatas e da Natureza, Universidade Federal de Pernambuco
 \newline
 Av. Prof. Moraes Rego, 1235, Cidade Universit\' aria, Recife, PE, Brazil
\newline
e-mail:  pablo@de.ufpe.br
}


\begin{document}
  

\begin{abstract}
Habitat loss is one of the biggest threats facing plant species nowadays.  We formulate a simple mathematical model of seed dispersal on reduced habitats to discuss survival of the species in relation to the habitat size and seeds production rate. Seeds get dispersed around the mother plant via several agents in a random way. In our model seeds landing sites are distributed according to a homogeneous Poisson point process with a constant rate on $\mathbb{R}$. We will assume that each seed will successfully germinate and grow into a new plant with the same characteristics as the mother plant. The time is discrete, scaled according to generations of plants or can represent years, since annual plants go through an entire growing cycle during one year. Then we will assume there are two symmetric barriers with respect to the origin and consider that the growth can not evolve past the barriers. Imposing barriers correspond to the physical limitation of the habitat. We appeal to tools of Probability Theory to formalize and study such a model, which can be seen as a discrete-time one-dimensional branching random walk with barriers. By means of coupling techniques and the comparison with suitably constructed multi-type branching processes we localize the critical parameter of the process around which there is survival with positive probability or extinction almost surely. In addition, we consider a discrete-space version of the model for which exact results are also obtained.    
\end{abstract}

\maketitle

\section{Introduction}

Accelerated climate change makes the humanity face many urgent issues. Persistence of species is one of them, especially of those inhabiting areas severely damaged by effects of global worming. The topic has been studied extensively in ecological literature, for example \cite{tejo2021, tejo2017, parmesan2006}. Among many factors affecting a species survival is its reproduction rate. The starting point of this work is the formulation of a simple mathematical model of seed dispersal as a phenomenon assisting reproduction in annual plants. Seeds get dispersed around the mother plant via several agents (wind, birds, water, etc) in a random manner. There are several studies relating dispersal to spatial random processes e.g. \cite{abdullahi2019, rebolledo2019}. The subject of interest there has been mostly competition of species whereas in this work we  focus on a survival of a species in relation to the habitat size and seeds production rate.
\par In our model seeds landing sites are distributed according to a homogeneous Poisson point process with a constant rate. This assumption has been widely used in ecological models like \cite{levine2002,clark1999,sagnard2007,ribbens1994}, among others. We will assume that each seed will successfully germinate and grow into a new plant with same characteristics as the mother plant. The time is discrete, scaled according to generations of plants or can represent years, since annual plants go through an entire growing cycle during one year.
\par Under these assumptions, we will focus on questions of survival and extinction of a species on an island. By  {\it islands} are not considered only the islands in the usual sense but  rather islands in the landscape, like mountaintops, oasis in the desert, grassland surrounded by houses, etc. This topic is related to the area of {Island bio-geography} that studies distribution of biodiversity over space and time of islands.  Research in this field started with MacArthur and Wilson in 1960's \cite{macarthur} and has been expanding since then. For a thorough review of ecological responses to recent climate change see \cite{parmesan2006} and references therein. 
\par We appeal to tools of Probability Theory to formulate a mathematical model of the seed dispersal. This stochastic model can be seen as a one-dimensional {\em branching random walk} (BRW). Roughly speaking, a branching random walk describes the evolution of particles living in a spatially structured environment, which give birth to new particles whose number and positions depend on a given reproduction law. This can be formulated as a discrete-time stochastic process whose space state is described by the collection of possible positions of particles at any time. For an overview of the formulation and recent results of these type of stochastic processes we refer the reader to \cite{BZ}. In our model, the space is continuous, and since it is a one-dimensional model, it is exactly $\mathbb{R}$. Moreover, as a reproduction law we use a Poisson point process associated to each particle so its realization represents the progeny of the particle. 
\par Extinction and survival of BRW, with respect to the dispersion rate, are well studied through branching processes \cite{athreya-ney}. Survival is the event of having at least one particle at any time of the process; extinction, of course, its complementary event. It is well-known that this process exhibits a {\em phase transition} phenomenon. There is a critical dispersal rate, $\lambda_c$, below which the process dies out almost surely (sub-critical case). Similarly, for the dispersal rates above the critical one (super-critical case), the survival is possible and in that case one can look into long-term spatial distribution of individuals. The critical density of a BRW, as described here, is known to be equal to $1$. In the super-critical case, a central limit behaviour is shown in \cite{bigginsclt}.
 \par The most relevant questions regarding the climate change context, ice melting, and {\em islands} shrinking look into the change of survival conditions i.e. how the critical dispersal rate change with introduction of spatial barriers. 
\par In our model we will assume there are two symmetric barriers with respect to the origin and consider a BRW on $\mathbb{R}$ that can not evolve past the barriers. We emphasize that imposing barriers to the BRW correspond to the physical limitation of the habitat. Initially there is an individual located at the origin. It produces a Poisson number of children that are uniformly distributed in its neighborhood of the unit size. The second generation is distributed as a Poisson Point Process on $[-1,1]$ with rate $\lambda/2$, so that the expected total number of children equals $\lambda$. Every new individual produces the offspring in its own neighborhood following the same law and independently of other siblings. 

\par Note that, in general, we are assuming that the process can not evolve past points $-L$ and $L$ ($L \in \mathbb{R}$). The main focus of our work is to study how the barriers affect phase transition in the BRW and change in critical density $\lambda_c$. Similar processes with only one barrier were studied in different settings by many authors like \cite{berard2011,barrier,BZR,derrida,jaffuel}. 

\par The rest of the paper is organized as follows. Our study is subdivided into two parts. The first one is the formulation and study of the BRW with barriers at $-1$ and $1$. In Section \ref{sec:model} we give the formal definition of the main process $(\mathcal{Y}_n)_{n\in\mathbb{N}}$ and subsequently state our main result, Theorem \ref{thm:bounds}, which allow us to localize the critical value. To prove this theorem, we appeal to the construction of two auxiliary multi-type branching processes $(\mathcal{X}_n)_{n\in\mathbb{N}}$ and $(\mathcal{Z}_n)_{n\in\mathbb{N}}$ which sandwich the original process. For the auxiliary processes we are able to obtain critical values numerically. The monotonic relation between these processes and a coupling argument is then used to find the critical density of $(\mathcal{Y}_n)_{n\in\mathbb{N}}$, see Corollary \ref{cor:main} and Theorem \ref{thm:main}. Such constructions and the related results are included in Section \ref{sec:branching}. The second part of our work is organized in Section \ref{sec:discrete}, where we propose a discrete-space version of the model. The advantage of this approach is that for this process we are able to obtain the critical value exactly, not only for $L=1$, and to see how it changes with $L$. Finally, Section \ref{sec:discussion} is devoted to a discuss of our results and prospects for future research. 

\section{The model and the existence of phase transition}\label{sec:model}
Initially there is one particle located in the origin. The descendants of this particle are scattered in the interval $[-L,L]$ according to a Poisson point process $\xi$ with rate $\lambda /2$ (written also as $\xi \sim$ PPP($\lambda/2)$), where $L \geq 1$ and $\lambda > 0$. The first generation is constituted by the descendants of the particle located at the origin. Each of the particles in the first generation produces its own descendants according to the following mechanism. Assume that a particle in the first generation is located at $x \in [-L,L]$. Then, this particle tries to give rise to particles scattered at $[x-L,x+L]$ according to a Poisson point process $\xi$ with rate $\lambda/2$. Only attempts inside $[-L,L]$ are considered successful. The second generation is given by the successful births and their parents are considered dead. This procedure is repeated indefinitely, but can possibly end if there is no particle alive in a generation. This process is called {\em branching random walk} (BRW) with two barriers at $-L$ and $L$ respectively. We will focus our discussion on the case $L=1$. See Figure \ref{fig:realization}.

\vspace{.5cm}
\begin{figure}[h!]
    \centering
    \subfigure[][{At $n=0$ there is one particle located at $0$. At $n=1$, its descendants are scattered in the interval $[-1,1]$ according to a Poisson point process with rate $\lambda/2$, $\lambda>0$; and they compose the first generation. Whenever a particle in the first generation is located at $x\in [-1,1]$ then, it tries to give rise to particles scattered at $[x-1,x+1]$ according to an independent Poisson point process with rate $\lambda/2$. Those attempts outside $[-1,1]$, identified here with a red circle, are not allowed.}]{
    
\begin{tikzpicture}[thick, scale=0.8, every node/.style={scale=0.8}]


\draw[gray] (-2,0) to (12,0);
\filldraw [black] (5,0) circle (1.5pt);
\node at (0,0) {$[$};
\node at (10,0) {$]$};
\node at (-3,0) {$n=0$};
\node at (0,-0.4) {\footnotesize$-1$};
\node at (5,-0.4) {\footnotesize$0$};
\node at (10,-.4) {\footnotesize $1$};


\draw [->,dashed,gray] (5,0.15) to [out=90, in=270]  (1.2,1.85);
\draw [->,dashed,gray] (5,0.15) to [out=90, in=270]  (4.8,1.85);
\draw [->,dashed,gray] (5,0.15) to [out=90, in=270]  (7.5,1.85);


\draw[gray] (-2,2) to (12,2);
\filldraw [black] (1.2,2) circle (1.5pt);
\filldraw [black] (4.8,2) circle (1.5pt);
\filldraw [black] (7.5,2) circle (1.5pt);
\node at (-3,2) {$n=1$};
\node at (0,2) {$[$};
\node at (10,2) {$]$};


\draw[gray] (-2,4) to (12,4);
\draw [red] (-1.2,4) circle (3.5pt);
\filldraw [black] (-1.2,4) circle (1.5pt);
\filldraw [black] (3,4) circle (1.5pt);
\filldraw [black] (0.7,4) circle (1.5pt);
\filldraw [black] (4,4) circle (1.5pt);
\filldraw [black] (6.7,4) circle (1.5pt);
\filldraw [black] (5.5,4) circle (1.5pt);
\filldraw [black] (8.7,4) circle (1.5pt);
\filldraw [black] (10.4,4) circle (1.5pt);
\draw [red] (10.4,4) circle (3.5pt);
\node at (-3,4) {$n=2$};
\node at (0,4) {$[$};
\node at (10,4) {$]$};

\draw [->,dashed,gray] (1.2,2.15) to [out=90, in=270]  (-1.2,3.85);
\draw [->,dashed,gray] (1.2,2.15) to [out=90, in=270]  (3,3.85);
\draw [->,dashed,gray,opacity=0.7] (4.8,2.15) to [out=90, in=270]  (0.7,3.85);
\draw [->,dashed,gray,opacity=0.7] (4.8,2.15) to [out=90, in=270]  (4,3.85);
\draw [->,dashed,gray,opacity=0.7] (4.8,2.15) to [out=90, in=270]  (6.7,3.85);
\draw [->,dashed,gray] (7.5,2.15) to [out=90, in=270]  (5.5,3.85);
\draw [->,dashed,gray] (7.5,2.15) to [out=90, in=270]  (8.7,3.85);
\draw [->,dashed,gray] (7.5,2.15) to [out=90, in=270]  (10.4,3.85);

\node at (5,5) {$\vdots$};

\end{tikzpicture}}\qquad

\subfigure[][{The stochastic process may be seen as a discrete-time branching random walk (BRW) restricted to $[-1,1]$. The random set $\mathcal{Y}_n$ is the set of particles (its positions) alive in generation $n$, and the BRW is the stochastic process $\mathcal{Y}:=\left(\mathcal{Y}_n\right)_{n \in \mathbb{N}}$.}]{

\begin{tikzpicture}[thick,scale=0.8, every node/.style={scale=0.8}]


\draw[gray] (-2,0) to (12,0);
\filldraw [black] (5,0) circle (1.5pt);
\node at (0,0) {$[$};
\node at (10,0) {$]$};
\node at (-3,0) {$n=0$};
\node at (0,-0.4) {\footnotesize$-1$};
\node at (5,-0.4) {\footnotesize$0$};
\node at (10,-.4) {\footnotesize $1$};




\draw[gray] (-2,2) to (12,2);
\filldraw [black] (1.2,2) circle (1.5pt);
\filldraw [black] (4.8,2) circle (1.5pt);
\filldraw [black] (7.5,2) circle (1.5pt);
\node at (-3,2) {$n=1$};
\node at (0,2) {$[$};
\node at (10,2) {$]$};


\draw[gray] (-2,4) to (12,4);
\filldraw [black] (3,4) circle (1.5pt);
\filldraw [black] (0.7,4) circle (1.5pt);
\filldraw [black] (4,4) circle (1.5pt);
\filldraw [black] (6.7,4) circle (1.5pt);
\filldraw [black] (5.5,4) circle (1.5pt);
\filldraw [black] (8.7,4) circle (1.5pt);
\node at (-3,4) {$n=2$};
\node at (0,4) {$[$};
\node at (10,4) {$]$};


\node at (5,5) {$\vdots$};
\end{tikzpicture}}
\caption{Illustration of a possible realization of the BRW with two barriers at $-1$ and $1$, respectively, and offspring given by a Poisson point process with intensity $\lambda /2$, $\lambda >0$. Particles are represented by black points.}\label{fig:realization}
\end{figure}
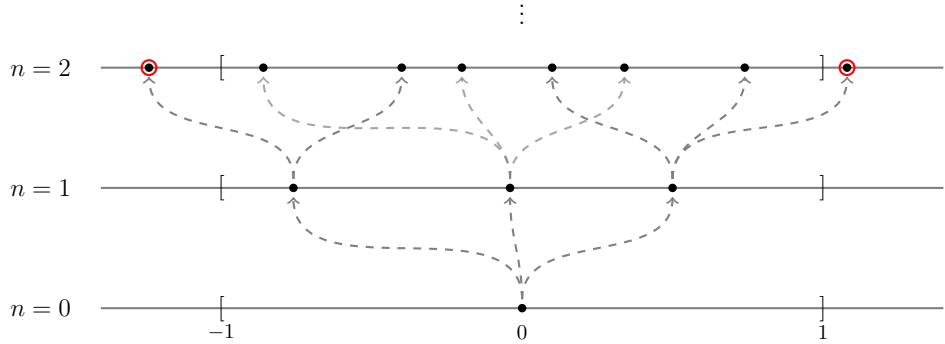
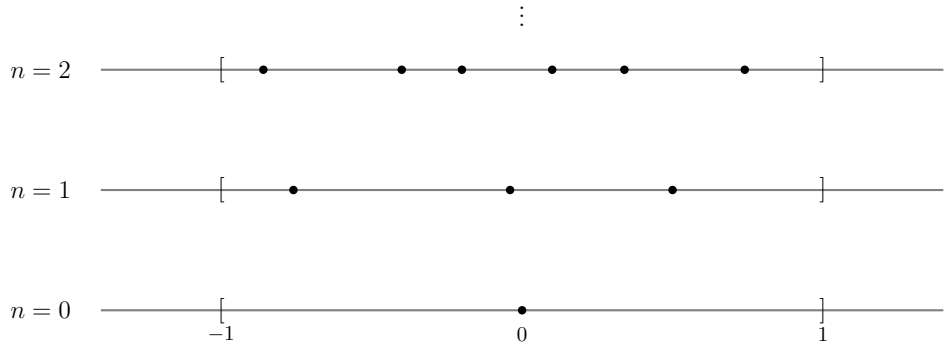

\vspace{.5cm}
Denote by $\mathcal{Y}_n$ the set of particles alive at generation $n$, for any $n\geq 0$, and let $\mathcal{Y}:=\left(\mathcal{Y}_n\right)_{n \in \mathbb{N}}$. We call $\mathcal{Y}$ the BRW with two barriers and offspring given by a Poisson point process with intensity $\lambda /2$. On the other hand, let $\mathcal{S}_{\lambda}$ be the event of survival of the process; that is,

$$\mathcal{S}_{\lambda}:=\bigcap_{n\geq 0}\{\mathcal{Y}_n\neq \emptyset\}.$$

Appealing to a coupling argument it is not difficult to see that the survival probability is non-decreasing on $\lambda$, i.e.
\begin{equation}
\mathbb{P}\left(\mathcal{S}_{\lambda_1}\right) \leq \mathbb{P}\left(\mathcal{S}_{\lambda_2}\right).    
\end{equation}
for any $0 < \lambda_1 < \lambda_2$. Here $\mathbb{P}$ stands for the law of the process. Therefore we may define the critical parameter for the BRW with two barriers $\mathcal{Y}$ as
\begin{equation}
\lambda_c\left(\mathcal{Y}\right) := \sup \{\lambda > 0 : \mathbb{P}\left(\mathcal{S}_{\lambda}\right) = 0 \} .
\end{equation}

Using multi-type branching processes in conjunction with coupling techniques we are able to localize the critical parameter $\lambda_c(\mathcal{Y})$ in the case $L=1$. We shall see that such a critical value is related to the Perron-Frobenious eigenvalue of a Toeplitz matrix. In what follows we use the notation $T_{k,d}$ for the $k\times k$ banded symmetric Toeplitz matrix with $0-1$ values and bandwidth $d$. For instance, 
$$T_{6,3}={\footnotesize \left[
\begin{array}{cccccc}
1&1&1&0&0&0\\
1&1&1&1&0&0\\
1&1&1&1&1&0\\
0&1&1&1&1&1\\
0&0&1&1&1&1\\
0&0&0&1&1&1\\
\end{array}
\right]}.$$

Now we can state the main results of our work.
\begin{theorem}\label{thm:bounds}
Let $\mathcal{Y}$ be the BRW with two barriers at $-1$ and $1$ respectively. Then
\[
\lim_{m\to \infty}\frac{2^{m+1}}{\rho\left(T_{2^{m+1},2^m}\right)}\leq \lc (\mathcal{Y}) \leq\lim_{m\to \infty}\frac{2^{m+1}}{\rho\left(T_{2^{m+1},2^m + 1}\right)}, 
\]
where, for $m\in \mathbb{N}$, $\rho\left(T_{k,d}\right)$ denotes the Perron-Frobenius eigenvalue of the $0-1$ values banded Toeplitz matrix $T_{k,d}$.
\end{theorem}

We refer the reader to \cite{ekstrom} for more details on banded Toeplitz matrices. Although there is no closed formula for the eigenvalues of such matrices of arbitrary dimension, Theorem \ref{thm:main} gains in interest if we realize that it allows us to obtain an approximation of the critical value from the numerical computation of the eigenvalues for several values of $m$.  

\begin{corollary}\label{cor:main}
Let $\mathcal{Y}$ be the BRW with two barriers at $-1$ and $1$ respectively. Then
\[
1.286907 \leq \lc (\mathcal{Y}) \leq  1.287096. 
\]
\end{corollary}

As it is well-known, the critical parameter at which  phase transition holds for the BRW without barriers is equal to $1$. In words, our result shows that imposing barriers to the original process produces a shift in its critical parameter to approximately $1.28$.  Coming back to our motivation, although we are dealing with a simplified model, this is enough to catch how the survival of a plant species is negatively affected in the presence of a reduced habitat.

\smallskip
In order to prove Theorem \ref{thm:bounds} we study two sequences of stochastic processes that {\em sandwich} $\cY$. Indeed, we say that the random set $A$ is dominated by the (random) set B if $A \subset B$ a.s. In this case we also say that $B$ dominates $A$. For any $m \geq 1$, the first sequence $(\cX^m_n)_{n\in\mathbb{N}}$ will be dominated by $\cY$ and the second one, denoted by $(\cZ^m_n)_{n\in\mathbb{N}}$ will dominate $\cY$. The sandwiching processes are related to multi-type branching processes with $2^{m+1}$ types. We then study critical densities of these multi-type branching processes as the number of types tends to infinity, which allows us to provide bounds for the critical parameter $\lc$ of the original process. Furthermore, we are able to prove that such bounds are equal. This is the content of the following corollary.

\begin{corollary}\label{thm:main}
Let $\mathcal{Y}$ be the BRW with two barriers at $-1$ and $1$ respectively. Then
\[
 \lc (\mathcal{Y}) =\lim_{m\to \infty}\frac{2^{m+1}}{\rho\left(T_{2^{m+1},2^m}\right)}=\lim_{m\to \infty}\frac{2^{m+1}}{\rho\left(T_{2^{m+1},2^m +1}\right)}.
\]
\end{corollary}


\section{Auxiliary results and proof of the results}\label{sec:branching}
We begin this section by describing the construction of two processes $\mathcal{X}^m:=\left(\mathcal{X}_n^m\right)_{n \in \mathbb{N}}$ and $\mathcal{Z}^m:=\left(\mathcal{Z}_n^m\right)_{n \in \mathbb{N}}$ indexed by $m\in\mathbb{N}$ and such that $\mathcal{X}_n^m \subseteq \mathcal{Y}_n \subseteq \mathcal{Z}_n^m$ almost surely for any $m \geq 1$ and for any $n \geq 0$. Then, we relate such processes to multi-type branching processes which will be used to obtain sequences of upper and lower bounds for $\lambda_c(\cY)$. Using properties of Hermitian matrices we prove that the critical parameter of our model coincides with the limit of (any) of such sequences. 

\subsection{Multi-type branching processes and a label for the process $\cY$}

In order to study the behavior of our model we appeal to the theory of discrete-time multi-type branching processes. In such processes and, just to fix ideas, particles can be classified into $k$ different types where $k \geq 1$ is fixed. After each time step, a particle of type $i$ will give birth to particles of different types according to a given probability law. Thus the multi-type branching process is a $k$-dimensional discrete-time  Markov chain $({\bf B}_{n})_{n\in \mathbb{N}}$, where ${\bf B}_{n}$ is the $k$-dimensional vector whose $ith-$coordinate, $i\in\{1,\ldots,k\}$, represents the number of particles of type $i$ which were given birth at time $n$. For a deeper discussion of these processes we refer the reader to \cite[Chapter V]{athreya-ney}. An object of interest when dealing with multi-type branching processes is the mean matrix $M = (m_{ij})_{i,j\in\{1, \ldots , k\}}$ where $m_{ij}$ denotes the expected number of type $j$ offspring of a single
type $i$ particle in one generation. Indeed, this matrix  carries information about the survival or extinction of the process. Here survival means that at all times at least one particle is alive, no matter its type. By \cite[Theorem 2, Chapter V]{athreya-ney} we know that there is survival with positive probability if, and only if, the maximum eigenvalue of $M$ is greater than $1$ provided the process $({\bf B}_{n})_{n\in \mathbb{N}}$ is positive regular and non-singular.  

In order to construct a multi-type branching processes related to the process $\cY$ consider the partition $\mathcal{P}_m$ of the interval $[-1,1]$ using $2^{m+1}+1$ equally distant points with fixed $m \geq 1$. That is, for $j\in\{0, \ldots, 2^{m+1}\}$ set $x^m_j = -1 + j/2^m$ and denote the associated partition by 

\begin{equation}\label{partition}
 \mathcal{P}_m := \{x^m_0,\ldots,x^m_{2^{m+1}}\}.
\end{equation}

For $1 \leq  j \leq 2^{m+1}-1$ let $I^m_j := [x^m_{j-1}, x^m_j)$, and set $I^m_{2^{m+1}} := [1-(1/2^{m}),1]$. Now label the particles of the original process $\mathcal{Y}$ according to their position relative to the partition $\mathcal{P}_m$. More precisely, for a given $m$, we say that a particle located at $x \in [-1,1]$ is of {\em type} $j$ if it belongs to $I^m_j$. The resulting process, with labeled particles, will be denoted by $\mathcal{Y}^m = \left(\mathcal{Y}^m_n\right)_{n \geq 0}$. The next subsections are devoted to the construction of the two processes $\mathcal{X}^m$ and $\mathcal{Z}^m$.

\subsection{The process $\mathcal{X}^m$}

 For any alive particle $y$ denote by $\mathcal{D}(y)$ its descendants. Observe that $\mathcal{D}(y)= [-1,1] \cap N_y$ where $N_y \sim$ PPP$(\lambda/2$) on $[y-1,y+1]$. Define $\mathcal{X}^m = \left(\mathcal{X}^m_n\right)_{n \geq 0}$ as follows. See Figure \ref{fig:procx}. 

\bigskip
\begin{enumerate}
 \item[i)] Initially let $\mathcal{X}^m_0 := \mathcal{Y}^m_0 = \{0\}$ and $\mathcal{X}^m_1 := \mathcal{Y}^m_1$.
\item[ii)] For $n \geq 1$, if $x \in \mathcal{X}_n^m$ is of type $j$, then
 \begin{enumerate}
     \item If $x<0$ then all its offspring in $[-1,x^m_{j-1} + 1]$ belong to $\mathcal{X}^m_{n+1}$.
     \item If $x \geq 0$ then all its offspring in $[x^m_j-1,1]$ belong to $\mathcal{X}^m_{n+1}$.
     \end{enumerate}
\end{enumerate}

\begin{figure}[h!]
\begin{center}
\begin{tikzpicture}[thick, scale=0.8, every node/.style={scale=0.8}]


\draw[gray] (-2,0) to (12,0);
\node at (2.48,0) {$)$};
\node at (2.52,0) {$[$};
\node at (4.98,0) {$)$};
\node at (5.02,0) {$[$};
\node at (7.48,0) {$)$};
\node at (7.52,0) {$[$};

\filldraw [black] (5,0) circle (1.5pt);
\node at (0,0) {$[$};
\node at (10,0) {$]$};
\node at (-3,0) {$n=0$};
\node at (0,-0.4) {\footnotesize$-1$};
\node at (2.5,-0.4) {\footnotesize$-1/2$};
\node at (5,-0.4) {\footnotesize$0$};
\node at (7.5,-0.4) {\footnotesize$1/2$};
\node at (10,-.4) {\footnotesize $1$};

\draw [->,dashed,gray] (5,0.15) to [out=90, in=270]  (1.2,1.85);
\draw [->,dashed,gray] (5,0.15) to [out=90, in=270]  (7.8,1.85);


\draw[gray] (-2,2) to (12,2);
\node at (2.48,2) {$)$};
\node at (2.52,2) {$[$};
\node at (4.98,2) {$)$};
\node at (5.02,2) {$[$};
\node at (7.48,2) {$)$};
\node at (7.52,2) {$[$};
\filldraw [black] (1.2,2) circle (1.5pt);
\filldraw [black] (7.8,2) circle (1.5pt);
\node at (-3,2) {$n=1$};
\node at (0,2) {$[$};
\node at (10,2) {$]$};

\node at (0.9,1.75) {$x$};
\node at (8.1,1.75) {$y$};


\draw[gray] (-2,4) to (12,4);
\draw [red] (-1.2,4) circle (3.5pt);
\node at (2.48,4) {$)$};
\node at (2.52,4) {$[$};
\node at (4.98,4) {$)$};
\node at (5.02,4) {$[$};
\node at (7.48,4) {$)$};
\node at (7.52,4) {$[$};
\filldraw [black] (-1.2,4) circle (1.5pt);
\filldraw [black] (3,4) circle (1.5pt);
\filldraw [black] (4.7,4) circle (1.5pt);
\filldraw [black] (3.5,4) circle (1.5pt);
\filldraw [black] (8.7,4) circle (1.5pt);
\filldraw [black] (10.4,4) circle (1.5pt);
\draw [red] (10.4,4) circle (3.5pt);
\draw [blue] (3.5,4) circle (3.5pt);
\node at (-3,4) {$n=2$};
\node at (0,4) {$[$}; 
\node at (10,4) {$]$};

\draw [->,dashed,gray] (1.2,2.15) to [out=90, in=270]  (-1.2,3.85);
\draw [->,dashed,gray] (1.2,2.15) to [out=90, in=270]  (3,3.85);
\draw [->,dashed,gray] (1.2,2.15) to [out=90, in=270]  (4.7,3.85);
\draw [->,dashed,gray,opacity=0.7] (7.8,2.15) to [out=90, in=270]  (3.5,3.85);
\draw [->,dashed,gray,opacity=0.7] (7.8,2.15) to [out=90, in=270]  (8.7,3.85);
\draw [->,dashed,gray,opacity=0.7] (7.8,2.15) to [out=90, in=270]  (10.4,3.85);

\node at (5,5) {$\vdots$};

\end{tikzpicture}
\end{center}
\caption{First steps in the definition of the process $\cX^m$. Here $m=1$ so $I_1^{1}=[-1,-1/2)$, $I_2^{1}=[-1/2,0)$, $I_3^{1}=[0,1/2)$ and $I_4^{1}=[1/2,1]$. At time $n=1$, the offspring of the particle located at the origin at $n=0$ is formed by the particles $x$, of type $1$, and $y$, of type $4$. Both particles have one attempt each of a birth not allowed by the barriers (red circles outside $[-1,1]$). Moreover, at $n=2$, while $x$ gave birth to two particles of type $2$, particle $y$ gave birth to one particle of type $4$ and, also, it has an attempt of birth which is not allowed by the construction of $\cX^1$ even if such attempt is inside $[-1,1]\cap[y-1,y+1]$ (blue circle outside $[0,1]$).}\label{fig:procx}
\end{figure}

\bigskip
Observe that by construction the process $\mathcal{X}^m$ is a subset of $\mathcal{Y}$ a.s. Thus the proof of the following lemma is an immediate consequence of this construction.

\begin{lemma} \label{mon}
Fix  $m \geq 1$. Then, for any $n \geq 1$,
\begin{equation}
{\mathcal{X}}^m_n \subseteq \mathcal{Y}^m_n\text{ a.s.}
\end{equation}
\end{lemma}



\par By ignoring the spatial position of the particles, $\cX^m$ may be seen as a regular discrete-time multi-type branching process with $2^{m+1}$ types. Indeed, if we denote by $X_{n}^{m,j}$ the number of particles of type $j$ in $\cX^m$ at time $n$ and if we let
$$
{\bf X}_n^m=(X_n^{m,1},\ldots,X_{n}^{m,2^{m+1}}),
$$
then the stochastic process ${\bf X}^m=({\bf X}_n^m)_{n\in\mathbb{N}}$ is a multi-type branching process with $2^{m+1}$ types, whose matrix of expected numbers of progeny of all types of parent particles of all types is given by

\begin{eqnarray}\label{Xmean}
 M(\cX^m)(i,j) = \frac{\lambda}{2} \cdot \frac{1}{2^m} \ind(|i-j|\leq 2^{m}-1),
\end{eqnarray}

where $\ind(A)$ denotes the $0-1$ random variable indicating the occurrence of the event $A$. That is $M(\cX^m)$ is a matrix of order $2^{m+1}$ matrix whose $(i,j)$-th entry represents  the mean number of children of type $j$ of an individual of type $i$. We denote by $\lambda_c\left(\mathcal{X}^m\right)$ the critical parameter of the process ${\bf X}^m$. The following result follows directly from Lemma \ref{mon}, and the fact that

$$|\cX^m_n|=\sum_{j=1}^{2^{m+1}} X_{n}^{m,j}.$$

\begin{proposition} \label{crit}
Fix $m\geq 1$. Then $\lambda_c\left(\mathcal{X}^m\right) \geq \lambda_c\left(\mathcal{Y}\right)$.
\end{proposition}


Note that for each $m$ the partition $\mathcal{P}_{m+1}$ is a refinement of the partition $\mathcal{P}_m$, i.e. $\mathcal{P}_m \subset \mathcal{P}_{m+1}$. Indeed, the partition using powers of $1/2$ is used exactly for this reason. Indeed, the natural attempt to slice $[-1,1]$ in intervals of length $1/m$ do not have this useful property. Now, suppose that there is a particle at $\xi < 0 $ whose type is $j$ in $\cX_n^m$. Then its offspring lie in the interval $[-1,x_{j-1}^m +1]$. For $m+1$ the same particle at $\xi$ reproduces either in the same interval 
$[-1,x_{j-1}^m +1]$ or in the larger interval $[-1,x_{j-1}^m +1+1/2^m]$ depending whether $\xi$ belongs to the first or the second half of the interval with endpoints $x_{j-1}$ and $x_j$. In any case, $|\cX^m_n| \leq |\cX^{m+1}_n|$ for every $m, n \geq 1$. Therefore

$$\lambda_c\left(\mathcal{X}^m\right) \geq \lambda_c\left(\mathcal{X}^{m+1}\right) \mbox{ for } m\geq 1.$$

The sequence of critical values is therefore non-negative, non-increasing, and bounded from above by $\lambda_c\left(\mathcal{X}^1\right) = 1.527864$ (see Table \ref{tab:critical}). Thus the sequence $\left(\lambda_c\left(\mathcal{X}^m\right)\right)_m$ has a limit. Combining this fact with  Proposition \ref{crit} we obtain
  \begin{eqnarray} \label{limineqX}
  \lambda_c(\mathcal{Y}) \leq \lim_{m \rightarrow +\infty} \lambda_c(\mathcal{X}^m).
  \end{eqnarray}




 


\subsection{The process $\mathcal{Z}^m$}

For coupling purposes we consider a version of the process $\mathcal{Y}$ constructed as follows. Assume that there is a particle alive at $y$. Denote by $M_y$ an homogeneous Poisson point process with intensity $\lambda /2$ on $[y-2,y+2]$. The offspring of $y$ are the Poisson points of $M_y$ which are at a distance no greater than one of $y$ and which are inside $[-1,1]$. Define $\mathcal{Z}^m = \left(\mathcal{Z}^m_n\right)_{n \geq 0}$ as follows. See Figure \ref{fig:procz}.

\bigskip
\begin{enumerate}
 \item[i)] Initially let $\mathcal{Z}^m_0 := \mathcal{Y}^m_0 = \{0\}$ and set $\mathcal{Z}^m_1 := \mathcal{Y}^m_1$.
 \item[ii)] For $n\geq 1$, if $z \in \mathcal{Z}_{n}^m$ is of type $j$, then 
 \begin{enumerate}
     \item If $z<0$ then take as its offspring the points of $M_z$ in $[-1,x^m_j + 1]$. These points belong to $\mathcal{Z}^{m}_{n+1}$.
     \item If $z \geq 0$ take as its offspring the points of $M_z$ in $[x^m_{j-1}-1,1]$. These points belong to $\mathcal{Z}^{m}_{n+1}$.
     \end{enumerate}
\end{enumerate}

\begin{figure}[h!]
\begin{center}
\begin{tikzpicture}[thick, scale=0.8, every node/.style={scale=0.8}]


\draw[gray] (-2,0) to (12,0);
\node at (2.48,0) {$)$};
\node at (2.52,0) {$[$};
\node at (4.98,0) {$)$};
\node at (5.02,0) {$[$};
\node at (7.48,0) {$)$};
\node at (7.52,0) {$[$};

\filldraw [black] (5,0) circle (1.5pt);
\node at (0,0) {$[$};
\node at (10,0) {$]$};
\node at (-3,0) {$n=0$};
\node at (0,-0.4) {\footnotesize$-1$};
\node at (2.5,-0.4) {\footnotesize$-1/2$};
\node at (5,-0.4) {\footnotesize$0$};
\node at (7.5,-0.4) {\footnotesize$1/2$};
\node at (10,-.4) {\footnotesize $1$};

\draw [->,dashed,gray] (5,0.15) to [out=90, in=270]  (1.2,1.85);
\draw [->,dashed,gray] (5,0.15) to [out=90, in=270]  (7.8,1.85);


\draw[gray] (-2,2) to (12,2);
\node at (2.48,2) {$)$};
\node at (2.52,2) {$[$};
\node at (4.98,2) {$)$};
\node at (5.02,2) {$[$};
\node at (7.48,2) {$)$};
\node at (7.52,2) {$[$};
\filldraw [black] (1.2,2) circle (1.5pt);
\filldraw [black] (7.8,2) circle (1.5pt);
\node at (-3,2) {$n=1$};
\node at (0,2) {$[$};
\node at (10,2) {$]$};

\node at (0.9,1.75) {$x$};
\node at (8.1,1.75) {$y$};


\draw[gray] (-2,4) to (12,4);
\draw [red] (-1.2,4) circle (3.5pt);
\node at (2.48,4) {$)$};
\node at (2.52,4) {$[$};
\node at (4.98,4) {$)$};
\node at (5.02,4) {$[$};
\node at (7.48,4) {$)$};
\node at (7.52,4) {$[$};
\filldraw [black] (-1.2,4) circle (1.5pt);
\filldraw [black] (3,4) circle (1.5pt);
\filldraw [black] (4.7,4) circle (1.5pt);
\filldraw [red!80!black] (7.3,4) circle (1.5pt);
\filldraw [black] (3.5,4) circle (1.5pt);
\filldraw [black] (8.7,4) circle (1.5pt);
\filldraw [black] (10.4,4) circle (1.5pt);
\draw [red] (10.4,4) circle (3.5pt);
\node at (-3,4) {$n=2$};
\node at (0,4) {$[$}; 
\node at (10,4) {$]$};

\draw [->,dashed,gray] (1.2,2.15) to [out=90, in=270]  (-1.2,3.85);
\draw [->,dashed,gray] (1.2,2.15) to [out=90, in=270]  (3,3.85);
\draw [->,dashed,gray] (1.2,2.15) to [out=90, in=270]  (4.7,3.85);
\draw [->,dashed,red!80!black,opacity=0.5] (1.2,2.15) to [out=50, in=270]  (7.3,3.85);
\draw [->,dashed,gray,opacity=0.7] (7.8,2.15) to [out=90, in=270]  (3.5,3.85);
\draw [->,dashed,gray,opacity=0.7] (7.8,2.15) to [out=90, in=270]  (8.7,3.85);
\draw [->,dashed,gray,opacity=0.7] (7.8,2.15) to [out=90, in=270]  (10.4,3.85);

\node at (5,5) {$\vdots$};

\end{tikzpicture}
\end{center}
\caption{{First steps in the definition of the process $\cZ^m$. Here $m=1$ so $I_1^{1}=[-1,-1/2)$, $I_2^{1}=[-1/2,0)$, $I_3^{1}=[0,1/2)$ and $I_4^{1}=[1/2,1]$. At time $n=1$, the offspring of the particle located at the origin at $n=0$ is formed by the particles $x$, of type $1$, and $y$, of type $4$. Both particles have one attempt each of a birth not allowed by the barriers (red circles outside $[-1,1]$). Moreover, at $n=2$, following the rules of the process $\cY$, $x$ gave birth to two particles of type $2$, and particle $y$ gave birth to one particle of type $2$ and other of type $4$. In addition, an additional birth is allowed for $x$ by the construction of $\cZ^1$ even if such attempt is outside $[-1,1]\cap[x-1,x+1]$ (red particle inside $[-1,1/2)$).}}\label{fig:procz}
\end{figure}

\bigskip
We analyze the process $\cZ^m$ in a similar way as we analyzed $\cX^m$. We may think of $\cZ^m$ as being a multi-type branching process. In this case, one can denote by $Z_{n}^{m,j}$ the number of particles of type $j$ in $\cZ^m$ at time $n$ and set
$${\bf Z}_n^m=(Z_n^{m,1},\ldots,Z_{n}^{m,2^{m+1}}).$$
Thus, ${\bf Z}^m=({\bf Z}_n^m)_{n\in\mathbb{N}}$ is a multi-type branching process with $2^{m+1}$ types. Similarly as in \eqref{Xmean} its matrix of expected values has the following entries

\begin{eqnarray}\label{Zmean}
 M(\cZ^m)(i,j) = \frac{\lambda}{2} \cdot \frac{1}{2^m} \ind(|i-j|\leq 2^{m}).
\end{eqnarray}

 The proof of the following lemma is an immediate consequence of the construction of $\mathcal{Z}^m_n$.

\begin{lemma}\label{lem:monz}
Fix $m\geq 1$.  Then, for any $n \geq 1$,
\begin{equation}
\mathcal{Y}^m_n \subset \mathcal{Z}^m_n \text{ a.s.}
\end{equation}
\end{lemma}

\begin{remark}
The reason for considering the Poisson process $M_y$ on $[y-2,y+2]$ is to guarantee that $\mathcal{Y}^m_n \subset \mathcal{Z}^m_n$ a.s. There are many other choices rather than 2 but this choice avoid introducing more cumbersome notation. Indeed, for fixed $m$ we may consider $[y - (1+1/2^m), y+(1+1/2^m)]$ instead of $[y-2,y+2]$ but this notation would introduce an unnecessary difficulty to the reader.
\end{remark}

We denote by $\lambda_c\left(\mathcal{Z}^m\right)$ the critical parameter of ${\bf Z}^m$. The following result follows directly from Lemma \ref{lem:monz} and the fact that

$$|\cZ_n^m|=\sum_{j=1}^{2^{m+1}} Z_{n}^{m,j}.$$

\begin{proposition}
For any $m\geq 1$,
$$\lambda_c\left(\mathcal{Z}^m\right) \leq \lambda_c\left(\mathcal{Y}\right).$$
\end{proposition}

Since $\mathcal{P}_{m + 1}$ is a refinement of $\mathcal{P}_m$, then $|\cZ^m_n| \geq |\cZ^{m+1}_n|$. Indeed, consider a particle located at $ \xi <0$ in the $\cZ^m$ process of type $j$. This particle reproduces in the interval $[-1,x_j^m +1]$. In the process $\cZ^{m+1}$, the particle at $\xi$ is either of type $j$ or type $j+1$. In the latter case it reproduces in the same interval as the $m-$th case, while in the former case that interval is shortened by $1/2^m$. Thus, the population of $\cZ^m$ does not increase  as $m$ grows and therefore

\[\lambda_c\left(\mathcal{Z}^m\right) \leq \lambda_c\left(\mathcal{Z}^{m+1}\right), \mbox{for } m \geq 1.\]

This sequence of critical values is non-decreasing and bounded from above by $\lambda_c\left(\mathcal{X}^1\right)$. Also, we have

\begin{eqnarray}\label{limineqZ}
 \lim_{m \rightarrow +\infty} \lambda_c(\mathcal{Z}^m) \leq \lambda_c(\mathcal{Y}).
\end{eqnarray}


 
\vskip3mm
In the next section we use (\ref{limineqX}) and (\ref{limineqZ}) in order to obtain bounds for $\lambda_c(\mathcal{Y})$.
 \vskip3mm
 
 \subsection{Proof of Theorem \ref{thm:bounds} and Corollary \ref{cor:main}}

By the previous results we obtain that 

\begin{equation}\label{eq:bounds} 
\lim_{m \rightarrow +\infty} \lambda_c(\mathcal{Z}^m) \leq \lambda_c(\mathcal{Y})\leq  \lim_{m \rightarrow +\infty} \lambda_c(\mathcal{X}^m) ,\end{equation}

\noindent
where $\lambda_c(\mathcal{X}^m)$ and $\lambda_c(\mathcal{Z}^m)$ are the critical parameters associated to the multi-type branching process related to the processes $\cX^m$ and $\cZ^m$, respectively. We point out that $\lambda_c(\mathcal{X}^m)$ and $\lambda_c(\mathcal{Z}^m)$ are the maximum eigenvalue of the mean values matrices for the respective multi-type branching processes, see \cite[Theorem 2, Chapter V]{athreya-ney}. Their entries are given by (\ref{Xmean}) and (\ref{Zmean}), respectively. In particular, for $m=1$, we get

\[M(\cX^1)= \frac{\lambda}{2} \times \frac{1}{4} \times
\begin{bmatrix}
1&1&0&0\\
1&1&1&0\\
0&1&1&1\\
0&0&1&1
\end{bmatrix},
     \hskip5mm  M(\cZ^1)= \frac{\lambda}{2} \times \frac{1}{4} \times
\begin{bmatrix}
1&1&1&0\\
1&1&1&1\\
1&1&1&1\\
0&1&1&1
\end{bmatrix},\]

\bigskip
\noindent
and for $m>1$ the mean values matrices are square matrices of order $2 \times 2^m$. Indeed, we have

\[  M(\cX^m)= \frac{\lambda}{2} \times \frac{1}{2^m} \times
{\footnotesize \left[
\begin{array}{ccccc!{\color{gray}\vrule}ccccc}
1&1&\cdots&1&1 &0&0&\cdots&0&0\\
1&1&\cdots&1&1&1&0&\cdots&0&0\\
\vdots & \vdots & \ddots& \vdots& \vdots& \vdots&\vdots&\ddots&\vdots& \vdots   \\
1&1&\cdots&1&1&1&1&\cdots&1&0\\
1&1&\cdots&1&1&1&1&\cdots&1&1\\\arrayrulecolor{gray}\hline
0&1&\cdots&1&1&1&1&\cdots&1&1\\
0&0&\cdots&1&1&1&1&\cdots&1&1\\
\vdots & \vdots & \ddots& \vdots& \vdots& \vdots&\vdots&\ddots&\vdots& \vdots   \\
0&0&\cdots&0&1&1&1&\cdots&1&1\\
0&0&\cdots&0&0&1&1&\cdots&1&1
\end{array}
\right]},
\]

and

\[  M(\cZ^m)= \frac{\lambda}{2} \times \frac{1}{2^m} \times
{\footnotesize \left[
\begin{array}{ccccc!{\color{gray}\vrule}ccccc}
1&1&\cdots&1&1 &1&0&\cdots&0&0\\
1&1&\cdots&1&1&1&1&\cdots&0&0\\
\vdots & \vdots & \ddots& \vdots& \vdots& \vdots&\vdots&\ddots&\vdots& \vdots   \\
1&1&\cdots&1&1&1&1&\cdots&1&1\\
1&1&\cdots&1&1&1&1&\cdots&1&1\\\arrayrulecolor{gray}\hline
1&1&\cdots&1&1&1&1&\cdots&1&1\\
0&1&\cdots&1&1&1&1&\cdots&1&1\\
\vdots & \vdots & \ddots& \vdots& \vdots& \vdots&\vdots&\ddots&\vdots& \vdots   \\
0&0&\cdots&1&1&1&1&\cdots&1&1\\
0&0&\cdots&0&1&1&1&\cdots&1&1
\end{array}
\right]},
\]



\bigskip
\noindent
where the black lines divide rows and columns into halves of size $2^{m}$. Observe that the $0-1$ matrices above are actually banded symmetric Toeplitz matrices. In what follows we use the notation $T_{k,d}$ for the $k\times k$ banded symmetric Toeplitz matrix with $0-1$ values and bandwidth $d$. Thus $M(\cX^1)= (\lambda/2) \times (1/4) \times T_{4,2}$, $M(\cZ^1)= (\lambda/2) \times (1/4) \times T_{4,3}$, and in general:

$$M(\cX^m)= \frac{\lambda}{2} \times \frac{1}{2^m} \times T_{2^{m+1},2^{m}}\,\,\,\text{ and }\,\,\,M(\cZ^m)= \frac{\lambda}{2} \times \frac{1}{2^m} \times T_{2^{m+1},2^{m}+1}.$$

The proof of Theorem \ref{thm:bounds} is completed upon observing that the critical values associated to any of the dominant process is computed using the formula:
\[
\lc(\cdot)= \frac{2^{m+1}}{\rho(M(\cdot))},
\]
where $\rho(M(\cdot))$ denotes the Perron-Frobenius eigenvalue of the Toeplitz matrix associated to $M(\cdot)$. From here, we use the Software R to obtain the eigenvalues numerically for several values of $m$. As the size of such matrices grow exponentially fast, we only compute these critical parameters up to $m=12$. In Table \ref{tab:critical} we list the  critical values for $\cX^m$ and $\cZ^m$, for $m \in \{1, \ldots, 12\}$. See also Figure \ref{fig:critical} for a graphical representation of these values as a function of $m$. 
\vskip5mm

\begin{table}[h!]
\centering{\small
\begin{tabular}{ccccccc}\toprule 
  
$m$ & $1$&$2$&$3$& $4$& $5$& $6$\\\midrule
 $\lc(\cX^m)$ &1.527864 & 1.393724 & 1.337647 & 1.311711 & 1.299210 & 1.293070  \\ [1mm]
 $\lc(\cZ^m)$ & 1.123106 &1.198501& 1.240855& 1.263415& 1.275074& 1.281004\\[.2cm]\toprule 
$m$ & $7$&$8$&$9$& $10$& $11$& $12$\\ \midrule
 $\lc(\cX^m)$ & 1.290027 & 1.288513 & 1.287757 & 1.287379 & 1.287191 & 1.287096\\[.1cm]
$\lc(\cZ^m)$ & 1.283995& 1.285496& 1.286249& 1.286625 & 1.286814 & 1.286907\\\bottomrule
 \end{tabular}
 }
\caption{Critical values for $\cX^m$ and $\cZ^m$, for $m \in \{1, \ldots, 12\}$.}\label{tab:critical} 
\end{table}

\vskip10mm

	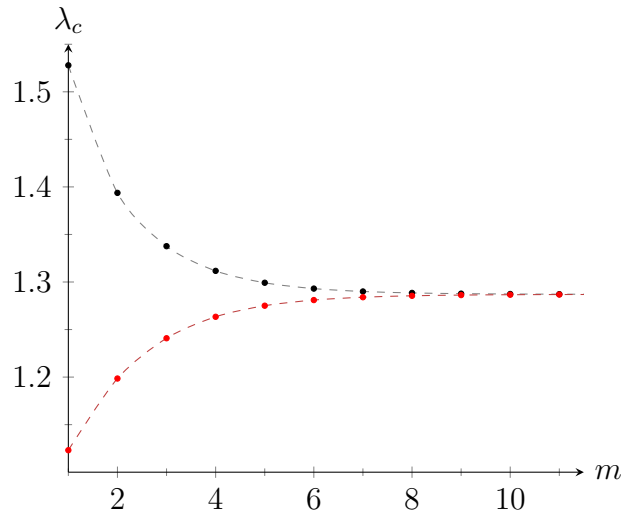
\begin{figure}[!htb]
	\pgfplotsset{my style/.append style={axis x line=middle, axis y line=
middle, xlabel={$m$}, ylabel={$\lambda_{c}$},every axis y label/.style={at=(current axis.above origin),anchor=south}, every axis x label/.style={at=(current axis.right of origin),anchor=west}}}
	\begin{tikzpicture}
\begin{axis}[my style,
xmin=1, xmax=11.5, ymin=1.1, ymax=1.55, minor tick num=1]
\addplot[
    color=gray,
dashed,
mark=*,
mark options={black,solid,mark size=1pt},
smooth
    ]
    coordinates {
    (1,1.527864)(2,1.393724)(3,1.337647)(4,1.311711)(5,1.299210)(6,1.293070)(7,1.290027)(8,1.288513)(9,1.287757)(10,1.287379)(11,1.287191)(12,1.287096)
    };
    
    \addplot[
    color=red!50!gray,
dashed,
mark=*,
mark options={red,solid,mark size=1pt},
smooth
    ]
    coordinates {
    (1,1.123106)(2,1.198501)(3,1.240855)(4,1.263415)(5,1.275074)(6,1.281004)(7,1.283995)(8,1.285496)(9,1.286249)(10,1.286625)(11,1.286814)(12,1.286908)
    };
\end{axis}

\end{tikzpicture}

		\caption{Comparison of $\lc(\cX^m)$ (black dots) and $\lc(\cZ^m)$ (red dots) as functions of $m$, for $m\in\{1,\ldots,12\}$.}\label{fig:critical}
	\end{figure}

Therefore we have the following numerical estimation for the critical value: 
\vskip3mm
\[  1.286907 \leq \lc \leq  1.287096,\]
\vskip3mm
which is the result stated in Corollary \ref{cor:main}. One can obtain better estimates just by computing the corresponding eigenvalues for $m > 12$. 

\bigskip
\begin{figure}[!htb]
	\pgfplotsset{my style/.append style={axis x line=middle, axis y line=
middle, xlabel={$m$}, ylabel={$\log(\lc(\cX^m)- \lc(\cZ^m)$)},every axis y label/.style={at=(current axis.above origin),anchor=south}, every axis x label/.style={at=(current axis.right of origin),anchor=west}}}
	\begin{tikzpicture}
\begin{axis}[my style,
xmin=1, xmax=11.5, ymin=-9, ymax=-1, minor tick num=1]
\addplot[
    color=gray,
dashed,
mark=*,
mark options={black,solid,mark size=1pt},
smooth
    ]
    coordinates {
   (1, -0.9044659)(2, -1.6336128)(3, -2.3351909)(4, -3.0304065)(5, -3.7240508)(6, -4.4173637)
  (7,-5.1106766)(8, -5.8034923)(9, -6.4969710) (10,-7.1901182)(11, -7.8832654)(12, -8.5737635)
 
    };

\end{axis}

\end{tikzpicture}

		\caption{Linear decay of log($\lc(\cX^m)$- $\lc(\cZ^m)$) with $m$, for $m\in\{1,\ldots,12\}$.}  \label{fig:logcritical}
	\end{figure}
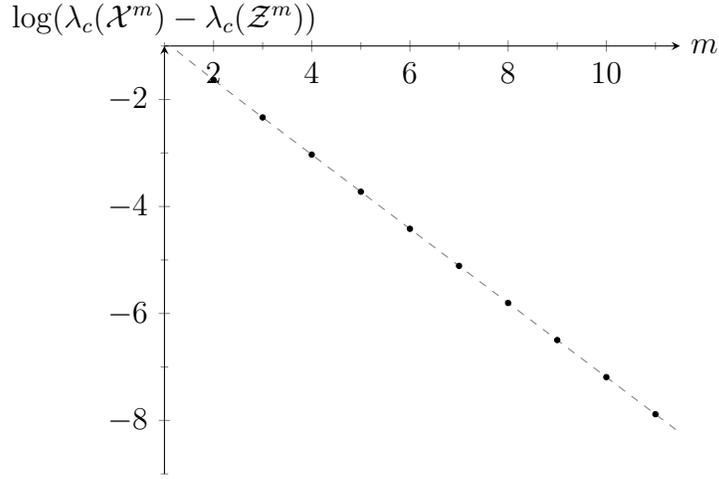

The obtained numerical values are clearly in support of the convergence of $\lc(\cX^m)$ and $\lc(\cZ^m)$ to a common limit. Moreover, Figure \ref{fig:logcritical} suggests that such convergence is exponentially fast. \newpage

\subsection{Proof of Corollary \ref{thm:main}}

By \eqref{eq:bounds} it is enough to prove that $\lim_{m \rightarrow +\infty} \lambda_c(\mathcal{Z}^m)\geq  \lim_{m \rightarrow +\infty} \lambda_c(\mathcal{X}^m) .$ Here $\lc(\cdot)= 2^{m+1}/\rho(\cdot),$ where $\rho(.)$ denotes the Perron-Frobenius eigenvalue of the Toeplitz matrix associated to $M(\cdot)$. Note that

$$M(\mathcal{X}^m) = M(\mathcal{Z}^m) - A^m$$

with

\[  A^m= 
{\footnotesize \left[
\begin{array}{ccccc!{\color{gray}\vrule}ccccc}
0&0&\cdots&0&0 &1&0&\cdots&0&0\\
0&0&\cdots&0&0&0&1&\cdots&0&0\\
\vdots & \vdots & \ddots& \vdots& \vdots& \vdots&\vdots&\ddots&\vdots& \vdots   \\
0&0&\cdots&0&0&0&0&\cdots&1&0\\
0&0&\cdots&0&0&0&0&\cdots&0&1\\\arrayrulecolor{gray}\hline
1&0&\cdots&0&0&0&0&\cdots&0&0\\
0&1&\cdots&0&0&0&0&\cdots&0&0\\
\vdots & \vdots & \ddots& \vdots& \vdots& \vdots&\vdots&\ddots&\vdots& \vdots   \\
0&0&\cdots&1&0&0&0&\cdots&0&0\\
0&0&\cdots&0&1&0&0&\cdots&0&0
\end{array}
\right]},
\]

\bigskip
\noindent
where the black lines divide rows and columns into halves of size $2^{m}$. These are Hermitian matrices so their eigenvalues are real and can be ordered as $\rho_{2^{m+1}}\le \rho_{n-1} \le \cdots \le \rho_1$. Although there is no general formula for the eigenvalues of a sum of Hermitian matrices, the Courant-Fischer theorem, see \cite{golub}, yields the lower bound:

$$\rho_{1}(M(\mathcal{Z}^m) - A^m)\geq \rho_{1}(M(\mathcal{Z}^m)) + \rho_{2^{m+1}}(-A^m).$$

\smallskip
Moreover, since $\rho_{2^{m+1}}(-A^m)=-1$, $\rho_{1}(M(\mathcal{Z}^m))=\rho(M(\mathcal{Z}^m))$, and $\rho_{1}(M(\mathcal{X}^m))=\rho(M(\mathcal{X}^m))$, we get

$$\lambda_c(\mathcal{X}^m)=\frac{2^{m+1}}{\rho(M(\mathcal{X}^m))}\leq \frac{2^{m+1}}{\rho(M(\mathcal{Z}^m))-1},$$
where the last inequality holds because $\rho(M(\mathcal{Z}^m))>1$ for any $m\in \mathbb{N}$ (indeed $\rho(M(\mathcal{Z}^m))\nearrow \infty$). Hence

\begin{equation}\label{eq:bounds2}
 \lim_{m \rightarrow +\infty} \lambda_c(\mathcal{X}^m)\leq  \lim_{m \rightarrow +\infty} \lambda_c(\mathcal{Z}^m).
 \end{equation}
 
This completes the proof. 

\section{A discrete-space version of the model}\label{sec:discrete}

It is worth pointing out that the main strategy to deal with our BRW with barriers model is the comparison with stochastic processes obtained from a kind of discretization of the original process. Such discretization comes from the classification of particles in a finite number of types. In this section, motivated by the construction of such auxiliary processes, we propose and study a related model on $\mathbb{Z}$ which allows us to obtain exact results for any $L\in \mathbb{N}$. Suppose that at time $0$ there is only one particle at the origin. The initial particle has Poisson($\lambda/3$) children at each position in $\{-1,0,1\}$ so the mean number of particles at generation $1$ is $\lambda$. In general, if a particle is located at site $k$, then it has Poisson($\lambda/3$) children at each position in $\{k-1,k,k+1\}$. Note that, thus defined, at the $n$-th generation there are $\lambda^n$ particles in total, in average. A question we are interested in is the distribution of the particles on $\{-n,...,n\}$. To formalize the model let $W_n(k)$ be the number of particles at site $k$ in the $n$-th generation, with $n\geq 0$ and $k\in\mathbb{Z}$. Let $\{ \mathcal{P}^l_{n,k}\}_{n,l,k}$ be a sequence of independent random variables with Poisson($\lambda/3$) distribution, for $n,l \in \mathbb{N}$, and $k \in \Z$. Then, we consider the stochastic process ${\bf W}=({\bf W_n})_{n\geq 0}$, with states-space $\{\mathbb{N}\cup \{0\}\}^{\mathbb{Z}}$, as follows:  
\begin{enumerate}
\item[(i)] $W_0(0)=1$ and $W_0(k)=0$, for $k \neq 0$.
\item[(ii)] $W_1(-1)=\mathcal{P}^1_{1,-1}$, $W_1(0)=\mathcal{P}^1_{1,0}$, and $W_1(1)=\mathcal{P}^1_{1,1} $.
\item[(iii)] For any $n>1$, we let for $k\in \mathbb{Z}$, 
$$W_{n+1}(k) = \sum_{i=1}^{W_n(k-1)} \mathcal{P}^i_{n,k-1}+\sum_{j=1}^{W_n(k)}\mathcal{P}^j_{n,k}+ \sum_{l=1}^{W_n(k+1)}\mathcal{P}^l_{n,k+1}.$$

\end{enumerate}
\vskip3mm

We point out that this is a discrete-space BRW. As a first result we characterize the expected number of particles at site $k$ in the $n$-th generation. In order to do it, we appeal to the trinomial triangle, which is a variation of Pascal's triangle such that an entry is the sum of the three entries above it:

\begin{equation}
    \centering
    \begin{array}{cccccccccc}
   &    &    &    &  1\\\noalign{\smallskip\smallskip}
    &    &    &  1 &  1  &  1\\\noalign{\smallskip\smallskip}
    &    &  1 &  2  &  3 &  2  &  1\\\noalign{\smallskip\smallskip}
    &  1 &  3  &  6 & 7   &  6 & 3   &  1\\\noalign{\smallskip\smallskip}
  1 &  4  &  10 & 16   &  19 &16    &  10 & 4   & 1\\\noalign{\smallskip\smallskip}
    \vdots &  \vdots  &  \vdots & \vdots   &  \vdots & \vdots    &  \vdots & \vdots   & \vdots\\\noalign{\smallskip\smallskip}
\end{array}
\end{equation}


\begin{proposition}\label{prop:trinomial-mean}Let $n\geq 0$ and $k\in \mathbb{N}$. Then $EW_n(k)=EW_n(-k)$, and
\begin{equation}
EW_n(k)=\left(\frac{\lambda}{3}\right)^n{n \choose k}_2,
\end{equation}
where $\displaystyle{n \choose k}_2$ denotes the $k$-th entry in the $n$-th row of the trinomial triangle. 
\end{proposition}

\begin{proof}
Let $m_n(k) := EW_n(k)$. Then $m_1(-1)=m_1(0)=m_1(1)= \lambda/3$ and we let $m_1=(\lambda/3)(1,1,1)$. Note that due to independence of the involved random variables we have, for $n=2$:

\begin{eqnarray*}
m_2(-2)&=& m_1(-1) m(-1)= 1\times \left(\frac{\lambda}{3}\right)^2,\\
m_2(-1)&=& m_1(-1) m(0)+m_1(0) m(-1)=2\times \left(\frac{\lambda}{3}\right)^2,\\
m_2(0)&=& m_1(-1) m(1) + m_1(0)m(0)+m_1(1)m(-1)= 3\times \left(\frac{\lambda}{3}\right)^2,
\end{eqnarray*}
and, because of symmetry, $m_2(1)= m_2(-1)$, and $m_2(2)= m_2(-2) $. Thus, the second generation in average looks like $m_2=(\lambda/3)^2 (1,2,3,2,1)$. Analogously, we obtain the third generation as $m_3=(\lambda/3)^3 (1,3, 6,7,6,3,1)$. For the sake of simplicity let's introduce the notation  $m_n=b_n \times(\lambda/3)^n $, and note that the first values of $b_n$ are given by:
\vskip5mm
\begin{center}
\begin{tabular}{rccccccccc}
$n=0$:&    &    &    &    &  1\\\noalign{\smallskip\smallskip}
$n=1$:&    &    &    &  1 &  1  &  1\\\noalign{\smallskip\smallskip}
$n=2$:&    &    &  1 &  2  &  3 &  2  &  1\\\noalign{\smallskip\smallskip}
$n=3$:&    &  1 &  3  &  6 & 7   &  6 & 3   &  1\\\noalign{\smallskip\smallskip}
$n=4$:&  1 &  4  &  10 & 16   &  19 &16    &  10 & 4   & 1\\\noalign{\smallskip\smallskip}
\end{tabular}
\end{center}
\vskip5mm
One can recognize here the Trinomial triangle, whose elements in the $n$-th row are coefficients in the expansion of $(1+x+x^2)^n$. Indeed, this can be proved by induction. The $k$-th entry in the $n$-th row is denoted by ${n \choose k}_2$ \cite{wolfram}. 
\end{proof}


Proposition \ref{prop:trinomial-mean} gains in interest if we realize that we have obtained exactly the shape of the mean vector $m_n$. Some explicit formulas for ${n \choose k}_2$ are given in \cite{wolfram}. This allows us to analyze the behavior of $EW_n(k)$, as a function of $n$, for some values of $k$ and different values of $\lambda$. See Figure \ref{fig:means} for a comparison of the mean number of particles at $0$ for $\lambda \in \{1.1,1.3,1.5\}$.

	\begin{figure}[!htb]
		\includegraphics{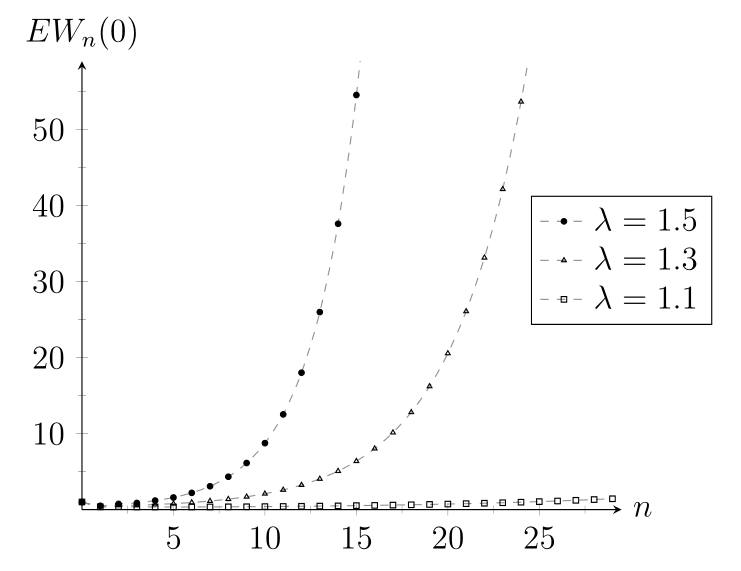}
		\caption{Behavior of the mean number of particles at $0$ as a function of $n$ for the discrete-space process with $\lambda \in\{1.1,1.3,1.5\}$.}\label{fig:means}
	\end{figure}

Now we shall see how that the process behaves when two barriers are imposed as in the BRW with barriers of the previous sections. Suppose that the discrete-space BRW defined above can not evolve outside the barriers $L$ and $-L$, for some $L\geq 1$, and denote it as ${\bf W^L}=({\bf W_n^L})_{n\geq 0}$. Note that now the states-space is given by $\{\mathbb{N}\cup \{0\}\}^{\{-L,\ldots,L\}}$. As in the continuous-space model we are interested in studying the survival of the process. Thus, let 

$$\mathcal{S}_{\lambda}(L):=\bigcap_{n\geq 0}\bigcup_{-L\leq k \leq L}\{W_n(k)\geq 1\},$$
be the event of survival of the process, and note that the survival probability is non-decreasing on $\lambda$. Therefore we may define the critical parameter for the process as
\begin{equation}
\lambda_c\left(L\right) := \sup \{\lambda > 0 : \mathbb{P}\left(\mathcal{S}_{\lambda}(L)\right) = 0 \} .
\end{equation}

The advantage of dealing with the discrete-space model is that we can obtain the exact localization of the critical parameter as a function of $L$. This is the result of our next theorem. As a consequence we can characterize the phase-diagram for this process, see Figure \ref{fig:critico-discreto}.  

\bigskip
	\begin{figure}[!htb]
	\subfigure[]{\includegraphics{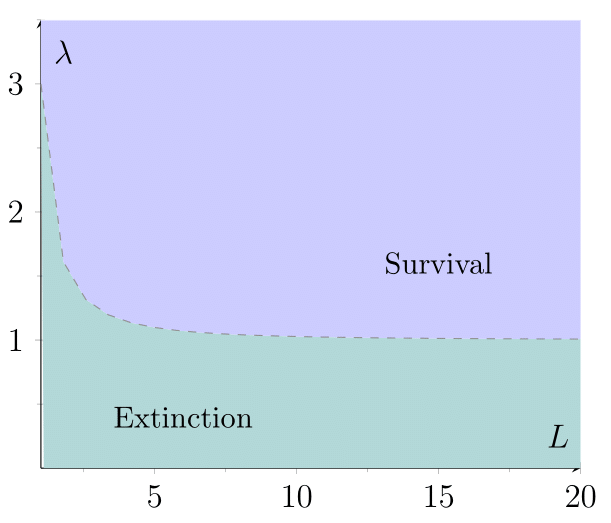}}
	\subfigure[]{\includegraphics{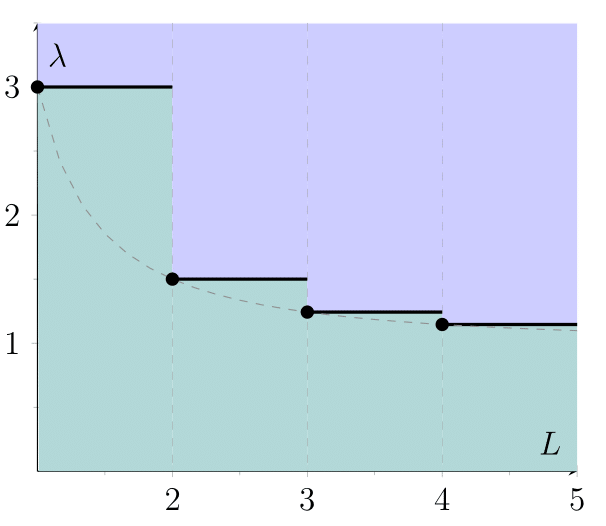}
	}

		\caption{Phase-diagram for the discrete-space process. The critical parameter is separating the behavior of the process between survival with positive probability, and almost surely extinction. (a) For illustration purposes we drew the critical parameter $\lambda_c(L)$ obtained in Theorem \ref{thm:crit-discrete} as a continuous function of $L$ for $L\in[1,\infty)$. (b) The exact step function $\lambda_c(L)$ for $L\in [1,\infty)$.}\label{fig:critico-discreto}
	\end{figure}


\begin{theorem}\label{thm:crit-discrete}
Let ${\bf W^L}$ be the discrete-space model with two barriers at $-L$ and $L$ respectively. Then
\[
\lc (L) = \displaystyle 3\left\{1+2\cos \displaystyle\left(\frac{\pi}{L+1}\right)\right\}^{-1}. 
\]
Moreover, as $ L \to \infty$, $\lambda_c(L) \to 1$, which is the critical value for the process without barriers.
\end{theorem}

\begin{proof}
As in the BRW with barriers we appeal to the theory of discrete-time multi-type branching processes. We consider the process for which particles can be classified into $2L+1$ different types where $L \geq 1$ is fixed. More precisely, for a given $k\in\{-L,\ldots,L\}$, we say that a particle located at $k$ is of {\em type} $k$. Then, it is not difficult to see that the process ${\bf W^L}$ is the discrete-time multi-type branching process with expected values given by: 
\[ M(i,j)= E(W_{n+1}(j)|W_{n}(i)=1))=\frac{\lambda}{3}~ \ind(|i-j| \leq 1).\]
Hence, the mean matrix, of dimension $2L+1 \times 2L+1$, is given by

\[
 M =
 \frac{\lambda}{3}\times {\small
 \begin{bmatrix}
  1 & 1 & 0&0 &\cdots&0&0 & 0& 0 \\
  1 & 1 & 1& 0&\cdots&0&0 & 0& 0 \\
  0&1&1&1&\cdots& 0 &0&0&0\\
  \vdots  & \vdots  &\vdots &\vdots & \ddots & \vdots& \vdots & \vdots & \vdots  \\
  0 &0 & 0 &0&\cdots &1 &1&1& 0\\
  0 &0 & 0 &0&\cdots &0 &1&1& 1\\
  0 & 0&0 &0&\cdots &0&0&1& 1
 \end{bmatrix} }= \frac{\lambda}{3} \times T_{2L+1,2}.
\]
Here $T_{2L+1,2}$ is a {\it tridiagonal} matrix $2L+1 \times 2L+1$ where all the non-zero values are equal to $1$. Note that although we are dealing again with a banded symmetric Toeplitz matrix, the spectra is known \cite{yueh}; namely, $\rho_k^L=1+2 \cos(k \pi/(2L+2)), k\in\{1,2,\ldots, 2L+1\}$. The Perron-Frobenius eigenvalue is $\rho_1^L= 1+2\cos(2\pi/(2L+2))$ and the corresponding eigenvector
\[ v^L = \left(\sin \frac{\pi}{2L+2}, \sin \frac{2\pi}{2L+2},...,\sin \frac{(2L+1)\pi}{2L+2}\right).\]

From here we have directly the critical value and also the limiting density. That is
\begin{eqnarray}
\lambda_c^L =\frac{3}{1+2\cos \frac{\pi}{L+1}}.
\end{eqnarray}
Note that $\lambda_c^L \to 1$ as $ L \to \infty$, which is the critical value for the process without barriers.


\end{proof}

\section{Discussion and future work} \label{sec:discussion}
In this work we propose a special stochastic process as a one-dimensional model of seed dispersal through space in a limited habitat or islands.  We focused on how shrinking of islands affects survival of a specie by studying the localization of the critical parameter for phase transition in the model. From an application point of view our arguments can be adapted to expand our findings to a two-dimensional model. In future work one could also consider a non-homogeneous Poisson point process to accommodate greater accumulation of seeds near the mother plant. Use of different density kernels in literature was reviewed in \cite{clark2001}. Both types of model extensions can be addressed by a suitable adaptation of our constructions to compare the original process with multi-type branching processes. 
\par It should be noted that our model, which can be seen as a branching random walk with two barriers, is of mathematical interest {\em{per se}}. As with the entire class of branching random walks, possible applications are multiple. In the context of computational models of evolution, the BRW with barriers was considered in \cite{bahar2013}. These models often include individual-based simulations in which organisms exist in a so-called morphospace. A point in that space represents the traits of an organism. The movement of the points in the space over time represent the evolutionary process in which groups of organisms may express varying traits as they evolve. For each generation, a population of offspring organisms was generated from the current (parent) population according to a reproduction scheme. After reproduction, the parent population was eliminated. Particles representing genetic traits differ from the mother particle through mutation and spacial barriers represent limits of viability. The findings of present work can also be applied to that setting.

 \section*{Acknowledgments}
 Part of this work was carried out during a visit of C.C. to ICMC-USP, and a visit of P.M.R. to UFABC. The authors are grateful to these institutions for their hospitality and support. Part of this work has been supported by Funda\c{c}\~ao de Amparo \`a Pesquisa do Estado de S\~ao Paulo - FAPESP (Grant 2017/10555-0). 

\end{document}